\pdfoutput=1
\RequirePackage{ifpdf}
\ifpdf 
\documentclass[pdftex]{sigma}
\else
\documentclass{sigma}
\fi

\usepackage{mathtools}

\usepackage[all]{xy}
\usepackage{tikz}

\newcommand{\N}{\mathbb{N}}
\newcommand{\Z}{\mathbb{Z}}
\newcommand{\Q}{\mathbb{Q}}
\newcommand{\R}{\mathbb{R}}
\newcommand{\C}{\mathbb{C}}

\newcommand{\bigslant}[2]{{\raisebox{.2em}{$#1$}\left/\raisebox{-.2em}{$#2$}\right.}}

\DeclareMathOperator{\conv}{conv}
\DeclareMathOperator{\vol}{vol}
\def \GL {\mathrm{GL}}
\def \Aff {\operatorname{Aff}}
\def \supp {\operatorname{supp}}

\def \calW {{\mathcal W}}
\def \calO {{\mathcal O}}

\def \tQ {\tilde{Q}}

\newtheorem{Theorem}{Theorem}[section]
\newtheorem*{Theorem*}{Theorem}
\newtheorem{Corollary}[Theorem]{Corollary}
\newtheorem{Lemma}[Theorem]{Lemma}
\newtheorem{Proposition}[Theorem]{Proposition}

\theoremstyle{definition}
\newtheorem{Definition}[Theorem]{Definition}
\newtheorem{Notation}[Theorem]{Notation}
\newtheorem{Example}[Theorem]{Example}
\newtheorem{Remark}[Theorem]{Remark}

\numberwithin{equation}{section}

\begin{document}

\allowdisplaybreaks

\newcommand{\arXivNumber}{2508.17551}

\renewcommand{\PaperNumber}{102}

\FirstPageHeading

\ShortArticleName{Integral Points and Volume of Integral-Integral Affine Manifolds}

\ArticleName{Integral Points and Volume\\ of Integral-Integral Affine Manifolds}

\Author{Oded ELISHA~$^{\rm a}$, Yael KARSHON~$^{\rm a}$ and Yiannis LOIZIDES~$^{\rm b}$}

\AuthorNameForHeading{O.~Elisha, Y.~Karshon and Y.~Loizides}

\Address{$^{\rm a)}$~School of Mathematical Sciences, Tel Aviv University, Israel}
\EmailD{\mail{odedelisha@mail.tau.ac.il}, \mail{yaelkarshon@tauex.tau.ac.il}}

\Address{$^{\rm b)}$~Department of Mathematical Sciences, George Mason University, USA}
\EmailD{\mail{yloizide@gmu.edu}}

\ArticleDates{Received February 17, 2026, in final form October 11, 2026; Published online October 29, 2026}

\Abstract{Using rational Ehrhart theory and Fourier analysis, we give an elementary proof that the volume of a closed manifold equipped with an integral-integral affine structure equals the number of its integral points.}

\Keywords{integral affine structure; Ehrhart theory}

\Classification{53C15; 52B20; 53D50}

\section{Introduction}
\label{sec:intro}

On an $n$-dimensional topological manifold $M$, an atlas
whose transition maps are locally of the form
\begin{equation} \label{eq:iia}
x \mapsto Ax +b
\qquad
\text{with $A \in \GL(n,\Z)$ and $b \in \Z^n$ }
\end{equation}
defines an \emph{integral-integral affine structure} \cite{hamilton_integral-integral_2024} (also known as a \emph{strongly integral affine structure}~\cite{sepe}).
This structure determines a (smooth) manifold structure on $M$, a measure on $M$, and a subset $M_\Z \subset M$ of \emph{integral points}.
For details see Definition~\ref{def:ii},
Remarks~\ref{rk:maximalii} and~\ref{measure_atlas},
and Definition \ref{def:MZ} below.

For example,
if $\Gamma$ is a group of transformations of $\R^n$
of the form \eqref{eq:iia} whose action on $\R^n$ is free and proper,
then the quotient $M := \R^n/\Gamma$ inherits an integral-integral
affine structure.
We~do not know if every closed integral-integral affine manifold is of this form;
this is a special case of the Markus conjecture \cite{markus}
(see the discussion in~\cite[Chapter 11]{goldman_geometric_2022}).

When $M$ is closed, its volume (total measure) is a positive real number and its number of integral points is a non-negative integer. In this paper, we give an elementary proof of the following result, first proved in \cite[Theorem~3.1]{hamilton_integral-integral_2024}.

\begin{Theorem} \label{main}
Let $M$ be a closed integral-integral affine manifold. Then its number
of integral points is equal to its volume.
\end{Theorem}

The motivation behind Theorem~\ref{main} comes from geometric quantization:
this theorem is a~key ingredient in an instance of the ``independence of polarization'' phenomenon for Lagrangian torus fibrations (see \cite[Section~2]{hamilton_integral-integral_2024}). In \cite{hamilton_integral-integral_2024}, Theorem~\ref{main} appeared as the ``downstairs theorem''; its proof was
surprisingly non-trivial and was inspired by ideas
that were sketched in papers on geometric quantization and on mirror symmetry
(see \cite[proof of Theorem 4.1]{andersen} and \cite[Conjecture~4.9]{gross}).
The proof involved intersection theory in a bundle of tori
over the integral-integral affine manifold and required a subtle cohomological argument \cite[Theorem~7.6]{hamilton_integral-integral_2024}.

In this paper, we provide a more elementary proof
of Theorem \ref{main}
that relies on rational Ehrhart theory in $\R^n$ \cite{sinai_robins,ehrhart} and on Fourier analysis.
Here is a sketch of our argument.
We~introduce the \emph{lattice-counting} function,
which associates to each $m \in \N$ the number $L_M(m)$ of points in $M$
whose coordinates are in $m^{-1}\Z$.
By rational Ehrhart theory, for a rational convex polytope $P$ in $\R^n$,
the number of points in $P \cap m^{-1}\Z^n$
is quasi-polynomial in $m$ (see Definition~\ref{quasi-polynomial}).
For a closed integral-integral affine manifold $M$,
we do not know how to decompose~$M$ into rational polytopes;
however, we can cover $M$ by overlapping rational polytopes and apply an inclusion-exclusion argument,
to conclude that $L_M(m)$ is quasi-polynomial in $m$.
On the other hand,
by ``decomposing'' $M$ using a smooth partition of unity supported in charts,
we conclude that $L_M(m) - \text{vol}(M) m^n = o(1)$
as $m \to \infty$.
From these two results, we conclude that $L_M(m) = \text{vol}(M) m^n$
on the nose (Theorem~\ref{main m}), yielding Theorem \ref{main} as the case $m=1$.

\section{Rational polytopes}

In this section, we introduce rational polytopes and the rational Ehrhart theorem.

\begin{Notation}
The set of natural numbers is
$\N \coloneq \{ 1, 2, \ldots \}$.
\end{Notation}

\begin{Notation}
We denote by $\Aff_{\Z}(\R^n)$
the set of maps $\R^n \to \R^n$ of the form
$x \mapsto Ax +b$ with $A \in \GL(n,\Z)$ and $b \in \Z^n$.
\end{Notation}

\begin{Definition}
\label{polytope_def}
 A set $P\subseteq \R^n$ is called a \textit{rational polytope} if there exist $v_1,\dots,v_k \in \Q^n$ such that
 $P=\conv\{v_1,\dots,v_k\}=\bigl\{\sum_{i=1}^k{\alpha_iv_i}\mid \alpha_i \ge 0,\, \sum_{i=1}^k{\alpha_i}=1\bigr\}$.
\end{Definition}
\begin{Definition}
Let $P\subseteq \R^n$ be a rational polytope. A point $x\in P$ is called a \textit{vertex} of $P$ if $x$ is an extreme point of~$P$.
\end{Definition}
\begin{Remark}
 If $P=\conv\{v_1,\dots,v_k\}$, then the set of vertices of $P$ is a subset of $\{v_1,\dots,v_k\}$.
\end{Remark}

\begin{Proposition}
 \label{affine_image_polytope}
 Let $P\subseteq \R^n$ be a rational polytope, and let $f\in \Aff_{\Z}(\R^n)$. Then $f(P)$ is a~rational polytope.
\end{Proposition}
\begin{proof}
 Let $P=\conv\{q_1,\dots,q_k\}$ for $q_i\in \Q^n$. Let $A \in \GL_n(\Z)$ and $b\in \Z^n$ be such that $f(x)=b+Ax$. Since $f$ is an affine map, $f(P)=\conv\{f(q_1),\dots,f(q_k)\}=\conv\{b+Aq_1,\dots,b+Aq_k\}$.
 ~Because $A \in \GL_n(\Z)$ and $b\in \Z^n$, we have $b+Aq_i \in \Q^n$ for all $1\le i\le k$. Thus, $f(P)$ is a~rational polytope.
\end{proof}

\begin{Proposition} \label{rational_intersection}
 Let $P_1, P_2 \subseteq \R^n$ be rational polytopes. Then $P_1\cap P_2$ is either empty or a~rational polytope.
\end{Proposition}
\begin{proof}
See \cite[Proposition 1.69]{gubeladze_polytopes_2009}.
\end{proof}

\begin{Definition} \label{quasi-polynomial}
A function $f \colon \Z \rightarrow \R$ is \textit{quasi-polynomial} if
there exist $n \in \N \cup\{0\}$
and functions $\{a_k \colon \Z \rightarrow \R\}_{k=0}^n$ such that
\begin{itemize}\itemsep=0pt
 \item each $a_k$ is periodic with integer period, and
 \item $f(m)=\sum_{k=0}^n{a_k(m)m^k}$, for all $m\in \Z$.
\end{itemize}
We say that a function $p \colon \N \to \R$ is quasi-polynomial
if it is the restriction to $\N$ of a quasi-polynomial function
$f \colon \Z \to \R$.
\end{Definition}

\begin{Proposition}\label{quasi_vanishes}
Let $p \colon \N \to \R$
be a quasi-polynomial function such that $p(m)=o(1)$ as $m \to \infty$. Then $p\equiv0$.
\end{Proposition}

\begin{proof}
By assumption, we can write
\[ p(m)=\sum_{k=0}^{n}{a_k(m)m^k} \]
for all $m \in \N$,
where each $a_k$ is a periodic function with some integer period $d_k$.
Let $d=\prod_{k=0}^{n}{d_k}$.
Then each $a_k$ is a periodic function with integer period $d$.
For each $1 \le i \le d$, the function
$g_i(m):=p(i+md)$ is then
(the restriction to $\N \cup \{0\}$ of)
a polynomial in~$m$
such that $g_i(m)=o(1)$ as $m \to \infty$.
So $g_i\equiv 0$ for every $1\le i\le d$,
and hence $p\equiv0$.
\end{proof}

\begin{Definition}
 \label{lattice_count_function}
 For a rational polytope $P\subseteq \R^n$, we define its \textit{lattice-counting function} by
 \begin{align*} &L_P \colon \ \N \rightarrow \N \cup \{0\}, \qquad
 L_P(m)\coloneq\#((m^{-1}\Z^n) \cap P) = \# \{x \in P \mid mx \in \Z^n \} .
 \end{align*}
\end{Definition}

\begin{Proposition}[rational Ehrhart] \label{rational_ehrhart}
 Let $P\subseteq\R^n$ be a rational polytope. Then its lattice-counting function is quasi-polynomial.
\end{Proposition}
\begin{proof} See \cite[Proposition 3.23]{sinai_robins}.
\end{proof}

\section{Integral-integral affine geometry}
\label{sec:prelim}

In this section, we define integral-integral affine manifolds, their integral points,
and rational polytopes in them.
We then prove that every closed integral-integral affine manifold can be covered by a finite collection of rational polytopes
whose intersections are either empty or rational polytopes;
see Corollary~\ref{rational_polytope_cover}.

\begin{Definition}
A map $f \colon \Omega \rightarrow \R^n$ on an open subset $\Omega \subseteq \R^n$ is called \textit{locally integral-integral affine} if
each point has a neighbourhood on which the map
agrees with an element of $\Aff_\Z(\R^n)$.
\end{Definition}

\begin{Definition} \label{def:ii}
 Let $M$ be a topological $n$-dimensional manifold.
 An atlas \[\{\varphi_i \colon U_i\rightarrow \Omega_i \subseteq \R^n\}_{i\in I}\] for $M$ is called an \textit{integral-integral affine atlas}
 if $\varphi_j\circ \varphi_i^{-1}$ is locally integral-integral affine
 on $\varphi_i(U_i \cap U_j)$ for every $i,j\in I$.
An~\textit{integral-integral affine structure} on $M$ is an integral-integral affine atlas that is maximal (with respect to inclusion of integral-integral affine atlases).
An~\textit{integral-integral affine manifold}
is a manifold equipped with an integral-integral affine structure.
 An \textit{integral chart} is an element of the maximal integral-integral affine atlas of $M$.
\end{Definition}

\begin{Remark} \label{rk:maximalii}
Every integral-integral affine atlas is contained in a unique maximal integral-integral affine atlas.
This follows from the fact that locally integral-integral affine
homeomorphisms between open subsets of $\R^n$ form a pseudogroup.
\end{Remark}

\begin{Remark} \label{measure_atlas}
 Let $M$ be an $n$-dimensional integral-integral affine manifold. Since locally integral-integral affine maps preserve the standard Lebesgue measure on $\R^n$, there exists a~unique {(Borel) measure} on $M$ such that each integral chart is measure preserving.
 The (affine) volume $\vol(M)$ is the measure of $M$ and integration on $M$ is done with respect to this measure.
\end{Remark}

\begin{Definition}\label{def:MZ}
Let $M$ be an $n$-dimensional integral-integral affine manifold. Its set of \textit{integral points} is
\begin{align*}
 M_{\Z}\coloneq \bigl\{x\in M \mid{} &\text{there exists
 an integral chart } \psi \colon U\rightarrow \Omega \subseteq \R^n \\
 &\text{such that $x \in U$ and } \psi(x)\in \Z^n\bigr\}.
\end{align*}
Moreover, for $m \in \N$, denote
\begin{align*}
 M_{m^{-1}\Z}\coloneq \bigl\{x\in M \mid {} & \text{there exists
 an integral chart } \psi \colon U\rightarrow \Omega \subseteq \R^n \\
 &\text{such that $x \in U$ and } \psi(x)\in m^{-1}\Z^n\bigr\}.
\end{align*}
\end{Definition}

\begin{Lemma}\label{if one then all}
For an $n$-dimensional integral-integral affine manifold $M$, and for $m \in \N$,
\begin{align*}
 M_{m^{-1}\Z} = \bigl\{x\in M \mid{} &\text{for every integral chart } \phi \colon U\rightarrow \Omega \subseteq \R^n \\
 & \text{such that $x \in U$, we have } \phi(x)\in m^{-1}\Z^n\bigr\}.
\end{align*}
\end{Lemma}

\begin{proof}
Let $x \in M_{m^{-1}\Z}$.
So there exists an integral chart $\psi \colon U'\rightarrow \Omega' \subseteq \R^n$
whose domain contains $x$ and such that $\psi(x)\in m^{-1}\Z^n$.
Let $\phi \colon U \to \Omega \subseteq \R^n$
be another integral chart whose domain contains~$x$.
By compatibility,
there exists a neighbourhood of $\psi(x)$
on which the composition
$\phi \circ \psi^{-1}$ has the form $y \mapsto b+Ay$ for $b\in \Z^n$ and $A\in \mathrm{GL}_n(\Z)$. Thus $m\phi(x)=m\phi\bigl(\psi^{-1}(\psi(x))\bigr)=m(b+A(\psi(x)))=mb+A(m\psi(x)) \in \Z^n$.
\end{proof}

\begin{Definition} \label{lattice count M}
 Let $M$ be an integral-integral affine manifold. Its \textit{lattice-counting function}~is
 \begin{align*}
 L_M \colon \ \N \rightarrow \N \cup \{0,\infty\},\qquad
 L_M(m)\coloneq\#M_{m^{-1}\Z} .
 \end{align*}
\end{Definition}

\begin{Proposition}{}
 Let $M$ be a closed $n$-dimensional integral-integral affine manifold. Then its lattice-counting function is finite: $L_M(m) < \infty$ for every $m\in \N$.
\end{Proposition}
\begin{proof}
 Assume not; let $(x_i)_{i=0}^\infty$ be a sequence of distinct points in $M_{m^{-1}\Z}$. Since $M$ is compact, this sequence has a converging subsequence $(x_{i_j})_{j=0}^\infty \rightarrow x^*$.
 Let $\phi\colon U \rightarrow \Omega\subseteq \R^n$ be an integral chart such that $x^* \in U$. Then for some $N>0$, we have $x_{i_j} \in U$ for all $j\ge N$. Then $\phi(x_{i_j})$ is a~convergent sequence of distinct elements of ${m^{-1}\Z^n}$, which is a contradiction.
\end{proof}

In order to use Ehrhart theory to study the lattice-counting function
of an integral-integral affine manifold $M$,
we cover $M$ by rational polytopes.
However, if we are not careful, nonempty intersections
of polytopes in the cover may be disconnected.

\begin{Definition}

A subset $P \subseteq M$
of an $n$-dimensional integral-integral affine manifold $M$
is called a \textit{rational polytope}
if there exists an integral chart $\phi$ defined in a neighbourhood of $P$
such that $\phi(P)\subseteq \R^n$ is a rational polytope.
\end{Definition}

\begin{Example}In the integral-integral affine manifold $\bigslant{\R}{4\Z}$,
the images of $[0,3]$ and of $[2,5]$ are rational polytopes,
and their intersection is disconnected.

\begin{center}
\begin{tikzpicture}[line cap=round]
\def\R{2} 
\def\Router{2.08}
\def\Rinner{1.92}

\draw[gray!60, line width=1pt] (0,0) circle (\R);

\draw[blue, line width=5pt]
 (0:\Router) arc[start angle=0, end angle=270, radius=\Router];

\draw[red, line width=5pt]
 (180:\Rinner) arc[start angle=180, end angle=450, radius=\Rinner];

\fill[blue] (0:\Router) circle (3pt);
\fill[blue] (270:\Router) circle (3pt);

\draw[red, line width=1.2pt] (180:\Rinner) circle (3pt);
\draw[red, line width=1.2pt] (450:\Rinner) circle (3pt);
\end{tikzpicture}
\end{center}
\end{Example}

In the definition of an integral-integral affine structure,
the transition maps are \emph{locally} integral-integral affine (see Remark~\ref{rk:maximalii}).
In Lemma~\ref{polytope_affine_chart} below,
we use a transition map that is \emph{globally} integral-integral affine.

\begin{Proposition}\label{locally_affine_connectedness}
Let $f \colon \Omega \rightarrow \R^n$ be a locally integral-integral affine map on an open subset~$\Omega$ of $\R^n$.
If $\Omega$ is connected,
then $f$ is the restriction to $\Omega$
of some element of $\Aff_\Z(\R^n)$.
\end{Proposition}

\begin{proof}Because $A\coloneq \mathrm{d}f|_{x} \in \GL_n(\Z)$ is locally constant, it is constant. So $b\coloneq f(x)-Ax \in \Z^n$ is locally constant, hence constant.
\end{proof}

\begin{Lemma}\label{polytope_affine_chart}
Let $M$ be an $n$-dimensional integral-integral affine manifold, $P\subseteq M$ a~rational polytope and $\psi \colon W \rightarrow \Omega \subseteq \R^n$ any integral chart such that $P\subseteq W$. Then $\psi(P) \subseteq \R^n$ is a~rational polytope.
\end{Lemma}
\begin{proof}
By definition, there is an integral chart $\phi \colon U \rightarrow \Omega'\subseteq \R^n$ such that $P \subseteq U$ and $\phi(P)$ is a rational polytope.
In particular, $\phi(P)$ is connected.
Let $\calO$ be the connected component of $\phi(U \cap W)$
that contains $\phi(P)$.
By Proposition \ref{locally_affine_connectedness},
there exists $f \in \Aff_\Z(\R^n)$
that coincides on~$\calO$ with $\psi \circ \phi^{-1}$.
By Proposition \ref{affine_image_polytope},
$\psi(P) = f(\phi(P))$ is a rational polytope.
\end{proof}

\begin{Lemma} \label{intersection_lemma}
 Let $M$ be an integral-integral affine manifold and $P_1,P_2 \subseteq M$ rational polytopes, and let $\phi_1\colon U_1\rightarrow \Omega_1\subseteq \R^n$ and $\phi_2\colon U_2 \rightarrow \Omega_2 \subseteq \R^n$ be integral charts such that $P_1\subseteq U_1$ and $P_2\subseteq U_2$. If $U_1\cap U_2$ is connected, then $P_1 \cap P_2$ is either a rational polytope or is empty.
\end{Lemma}

\begin{proof}
By Lemma~{\ref{polytope_affine_chart}},
$\phi_1(P_1)$ and $\phi_2(P_2)$ are rational polytopes.
The intersection $U_1\cap U_2$ is connected, so by
Proposition~\ref{locally_affine_connectedness},
the composition $\phi_1\circ\phi_2^{-1}$ has the form $x \mapsto Ax +b$ where $b\in \Z^n$ and $A\in \mathrm{GL}_n(\Z)$.
By Propositions~\ref{rational_intersection} and~\ref{affine_image_polytope}, the set
\[
Q\coloneq\phi_1(P_1)\cap (b+A\phi_2(P_2)) ,
\]
is a rational polytope. Moreover,
\begin{align*}
&\phi_1(P_1 \cap P_2) = \phi_1(P_1 \cap P_2 \cap U_1 \cap U_2) = \phi_1(P_1)\cap\phi_1(P_2\cap U_1 \cap U_2),
\end{align*}
and
\begin{align*}
 \phi_1(P_2\cap U_1 \cap U_2)&=\bigl(\phi_1\circ\phi_2^{-1}\bigr)(\phi_2(P_2\cap U_1 \cap U_2))=b+A(\phi_2(P_2)\cap \phi_2(U_1\cap U_2)) \\
 &=b+(A(\phi_2(P_2))\cap A\phi_2(U_1\cap U_2))\\ &=(b+A\phi_2(P_2))\cap(b+A\phi_2(U_1\cap U_2)).
 \end{align*}
 Therefore,
 \begin{align*}
 \phi_1(P_1 \cap P_2)=(\phi_1(P_1)\cap (b+A\phi_2(P_2)))\cap (b+A\phi_2(U_1\cap U_2))=Q\cap (b+A\phi_2(U_1\cap U_2)) .
 \end{align*}

The set $\phi_1(P_1\cap P_2)$ is closed in $Q$, because
$P_1$ and $P_2$ are compact.
It is also open in $Q$, because $b+A\phi_2(U_1\cap U_2)$ is open in $\R^n$.
Because $Q$ is (convex, hence) connected, $\phi_1(P_1\cap P_2)$
is either all of $Q$ or is empty.
\end{proof}

By repeated application of Lemma~\ref{intersection_lemma}, we obtain the following corollary.

\begin{Corollary} \label{finite_rational_intersection}
 Let $M$ be an integral-integral affine manifold, let $\{P_i\}_{i\in I}$ be a finite collection of rational polytopes in $M$,
 and let $\{\phi_i \colon U_i\rightarrow \Omega_i \subseteq \R^n\}_{i\in I}$ be integral charts such that $P_i\subseteq U_i$ for every $i \in I$. Suppose that, for every subset $I' \subseteq I$, the intersection $\bigcap_{i \in I'} U_i$ is connected.
 Then $\bigcap_{i \in I}{P_i}$ is either a rational polytope or is empty.
\end{Corollary}

\begin{Corollary} \label{rational_polytope_cover}
 Let $M$ be a closed integral-integral affine manifold. Then there exists a finite collection $\{P_i\}_{i=1}^k$ of rational polytopes such that
 \begin{enumerate}\itemsep=0pt
 \item[$(1)$] $M=\bigcup_{i=1}^k{P_i}$,
 \item[$(2)$] for every $I\subseteq \{1,\dots,k\}$, the intersection $\bigcap_{i\in I}{P_i}$ is either a rational polytope or is empty.
 \end{enumerate}
\end{Corollary}

\begin{proof}Recall that a \emph{good cover} of $M$
is an open cover where every nonempty finite intersection
is diffeomorphic to $\R^n$,
and that every open cover of $M$ has a refinement
that is a good open cover of $M$
\cite[Propositions 5.1 and 5.2]{bott_differential_1982}.

Thus, there exists a good open cover $\calW$ of $M$
such that every element of $\calW$ is contained in the domain
of an integral chart.

For each point $x \in M$, let $W_x$ be an element of $\calW$ that contains $x$,
let $\varphi \colon U \to \Omega$ be an~integral chart
whose domain contains $W_x$,
and let $\tQ_x \subseteq \R^n$ be a~closed cube with rational vertices
that is contained in $\varphi(W_x)$
and whose interior contains $\varphi(x)$.
Then $Q_x\coloneq \varphi^{-1}(\tilde Q_x)$ is a rational polytope in $M$.

Because $M$ is compact, there exist $x_1,\ldots, x_k \in M$
such that
$M=\bigcup_{i=1}^k{\text{Int}(Q_{x_i})}$.

For each $i \in \{ 1, \ldots, k \}$, let $P_i\coloneq Q_{x_i}$.
Then $M=\bigcup_{i=1}^k{P_i}$.
Because $\calW$ is a good cover,
for each $I' \subseteq \{1,\ldots,k\}$,
the intersection
$\bigcap_{i\in I'} W_{x_i}$ is connected.
By Corollary~\ref{finite_rational_intersection},
each intersection $\bigcap_{i\in I}{P_i}$ is either a rational polytope
or is empty.
\end{proof}

\section{Main result}

In this section, we calculate the lattice-counting function
of a closed integral-integral affine manifold;
see Theorem~\ref{main m}.

\begin{Proposition}[Ehrhart theorem for closed integral-integral affine manifolds]
 \label{manifold_ehrhart}
 The lattice-counting function of a closed integral-integral affine manifold is
quasi-polynomial.
\end{Proposition}

\begin{proof}
 Let $M$ be a closed $n$-dimensional integral-integral affine manifold and let $\{P_i\}_{i=1}^{k}$ be a rational polytope cover from Corollary~\ref{rational_polytope_cover}.
For each set of indices $I\subseteq \{1,\dots,k\}$,
denote $P_I\coloneq\bigcap_{i\in I}P_i$.
 By inclusion-exclusion,
 \begin{align*}
 L_M(m)=\sum_{x\in M_{m^{-1}\Z}}{1}=\sum_{j=1}^k{(-1)^{j+1}\sum\limits_{\substack{
 I\subseteq\{1,\dots,k\}
 \\ |I|=j}
 }{\# (M_{m^{-1}\Z}\cap P_I)}} .
 \end{align*}
 For every $I\subseteq \{1,\ldots,k\}$, take an integral chart $\psi\colon U\rightarrow \Omega\subseteq \R^n$ such that $P_I\subseteq U$. By~Lem\-ma~\ref{polytope_affine_chart}, $\psi(P_I)$ is a rational polytope. By Lemma~\ref{if one then all},
 \begin{align*}
 \#(M_{m^{-1}\Z}\cap P_I)=\#\bigl(m^{-1}\Z^n\cap\psi(P_I)\bigr).
 \end{align*}
By this and Proposition~\ref{rational_ehrhart},
$L_M$ is an alternating sum of quasi-polynomials, hence a quasi-polynomial.
\end{proof}

The fact that the lattice-counting function is quasi-polynomial allows
us to use asymptotic methods.

\begin{Proposition}{}
 \label{riemman_sum}
 Let $f\colon \R^n \rightarrow \R$ be a compactly supported smooth function. Then
 \begin{align*}
 \sum_{x\in m^{-1}\Z^n}{f(x)} = m^n\int_{\R^n}{f(x)\mathrm{d}x} + o(1)
 \end{align*}
 as $m \to \infty$.
\end{Proposition}

\begin{proof}
Because $f$ is a compactly supported smooth function, it is a Schwartz function.
Because the Fourier transform maps Schwartz functions to Schwartz functions \cite[p.~181, Corollary 2.2]{stein_fourier_2003},
the Fourier transform
\[
\hat{f} \colon\ \R^n \to \C,
\]
defined by
\[
\hat{f}(\zeta)=\int_{\R^n}{f(x)\mathrm{e}^{-2\pi \mathrm{i}\langle x,\zeta \rangle}\mathrm{d}x},\] is also a Schwartz function.
In particular, for every polynomial $P \in \R[y_1,\dots,y_n]$,
\begin{align*}
 \sup_{y\in \R^n}\bigl|\hat{f}(y)P(y)\bigr|< \infty .
\end{align*}
Hence, the series
$\sum_{y\in m\Z^n}{\hat{f}(y)}$ is absolutely convergent.

By the Poisson summation formula \cite[p. 88, Theorem~3.44]
{robins_fourier},
\begin{equation} \label{poisson}
 \sum_{x\in m^{-1}\Z^n}{f(x)}=m^n\sum_{y\in m\Z^n}{\hat{f}(y)}=m^n\biggl(\int_{\R^n}f(x)\mathrm{d}x+\sum_{0\neq y\in m \Z^n}\hat{f}(y)\biggr) .
\end{equation}
Thus we would like to prove that
 \begin{align}
 \label{remainder_exp}
 \sum_{0\neq y\in m \Z^n}\hat{f}(y) = o(m^{-n})
 \end{align}
 as $m \to \infty$.
(Although $\hat{f}$ is complex-valued,
by \eqref{poisson}
the sum in \eqref{remainder_exp} is real.)

For any positive integers $r \geq m$, set
 \[
 A_r^m\coloneq\{y\in m\Z^n \mid r \le |y| < r+1\}.
 \]
Then \[
 \# A_r^m \le \#\{y\in \Z^n \mid
 \text{$|y_j| \leq r$ for all $j$}
 \}
 \le (2r+1)^n \le (3r)^n,
\] and so, for any integer $N$ such that $2N-2 \geq n$,
\begin{equation} \label{annulus}
\sum_{y \in A_r^m} \frac{1}{|y|^{2N}}
 \leq (3r)^n \frac{1}{r^{2N}}
 = 3^n \cdot \frac{r^n}{r^{2N-2}} \cdot \frac{1}{r^2}
 \leq 3^n \cdot m^{n+2-2N} \cdot \frac{1}{r^2}.
\end{equation}
Because \[
\bigcup_{r=m}^\infty{A_r^m}=m\Z^n \setminus \{0\},
\] we have
 \begin{align*}
 &\biggl|\sum_{0\neq y\in m \Z^n}{\hat{f}(y)}\biggr| \le
 \sum_{0\neq y\in m \Z^n}{|\hat{f}(y)|} = \sum_{0\neq y\in m \Z^n}{\bigl|\hat{f}(y)|y|^{2N}\bigr|\frac{1}{|y|^{2N}}} \\ &\stackrel{\hat{f}\text{ is Schwartz}}\le C\cdot \sum_{0\neq y\in m \Z^n}{\frac{1}{|y|^{2N}}}
 =C \cdot \sum_{r=m}^{\infty}{\sum_{y\in A_r^m}{\frac{1}{|y|^{2N}}}} \stackrel{\eqref{annulus}}
 \le m^{n+2-2N} \cdot 3^nC\cdot \sum_{r=m}^\infty{\frac{1}{r^2}} \\
 & = o\bigl(m^{n+2-2N}\bigr) \text{ as $m \to \infty$} .
 \end{align*}
To obtain \eqref{remainder_exp}, take $N = n+1$.
\end{proof}

\begin{Remark}
 Although we do not require it, we have actually shown that
 \[
 \sum_{x\in m^{-1}\Z^n}{f(x)} = m^n\int_{\R^n}{f(x){\rm d}x} + o(m^{-\infty})
 \]
 as $m\to\infty$.
\end{Remark}

Combining the above asymptotics with the fact that $L_M$ is quasi-polynomial (Proposition~\ref{manifold_ehrhart}), we obtain our main result.

\begin{Theorem} \label{main m}
 Let $M$ be an $n$-dimensional closed integral-integral affine manifold. Then
 its lattice-counting function is
 \[
 L_M(m)=\vol(M)m^n.
 \]
\end{Theorem}

\begin{proof}
Because $M$ is locally compact, it has an integral-integral affine atlas
whose chart domains have compact closures in $M$.
Because $M$ is compact, it has such an atlas that is finite.
Let $\{\rho_i\}_{i=1}^{k}$ be a partition of unity
that is subordinate to the cover of $M$ by the domains of the charts
in such an atlas.
For each $i \in \{1,\ldots,k\}$,
let $\psi_{i} \colon U_{i} \to \Omega_{i} \subseteq \R^n$
be an integral chart whose domain contains
the support $\supp \rho_i$ of $\rho_i$ in $M$.
Then the subset $\psi_{i}(\supp \rho_i)$ of $\Omega_i$ is compact,
so $\rho_i \circ \psi_{i}^{-1}$
extends by zero to a compactly supported
smooth function on~$\R^n$.
By~Proposition~\ref{riemman_sum}, we have (see Remark~\ref{measure_atlas})
\begin{align*}
\sum_{x\in M_{m^{-1}\Z}}{\rho_i(x)}&=\sum_{y\in m^{-1}\Z^n \cap \psi_{i}(U_{i})}{\rho_i\bigl(\psi^{-1}_{i}(y)\bigr)}=m^{n}\int_{\psi_i(U_{i})}{\bigl(\rho_i\circ\psi^{-1}_{i}\bigr)(y){\rm d}y}+o(1)\\
 &
 =m^n\int_{M}{\rho_i(x){\rm d}x}+o(1).
\end{align*}
 Thus
 \begin{align*}
L_M(m)&=\#M_{m^{-1}\Z}=\sum_{x\in M_{m^{-1}\Z}}{1}=\sum_{x\in M_{m^{-1}\Z}}{\sum_{i=1}^{k}{\rho_i(x)}}=\sum_{i=1}^{k}{\sum_{x\in M_{m^{-1}\Z}}{\rho_i(x)}}\\
 &=\sum_{i=1}^{k}{\left(m^n\int_M{\rho_i(x){\rm d}x}+o(1)\right)}
 =m^n\left(\sum_{i=1}^{k}{\int_M{\rho_i(x){\rm d}x}}\right)+o(1)\\
 &=\vol(M)m^n+o(1).
 \end{align*}
By Proposition~\ref{manifold_ehrhart}, $L_M$ is quasi-polynomial; by Proposition~\ref{quasi_vanishes}, we obtain the desired equali\-ty.
\end{proof}

\begin{Remark}
Asymptotic expansions of a similar nature appeared in \cite{berline_local_2016, guillemin,loizides_semi-classical_2021}.
The idea to take advantage of quasi-polynomial behavior to show that an asymptotic result is exact appeared
in \cite{meinrenken_riemann-roch_1996}.
In the future, it might be interesting to seek generalizations
to integral-integral affine manifolds
of Euler--Maclaurin formulas relating sums to integrals \cite{berline_local_2016,karshon_sternberg}.
\end{Remark}

Finally, the case $m=1$ of Theorem~\ref{main m} yields Theorem~\ref{main}.

\subsection*{Acknowledgements}
We are grateful to Roei Raveh for providing
the idea of the proof of Proposition~\ref{riemman_sum}.
We thank the anonymous referees
for their helpful comments on an earlier version of this paper.

Y.~Karshon's research is partly funded
by the Natural Sciences and Engineering Research Council of Canada Discovery Grant RGPIN-2024-05798
and by the United States -- Israel Binational Science
Foundation Grant 2021730.
O.~Elisha is partly supported by Yaron Ostrover's
ISF Grant 938/22.

\pdfbookmark[1]{References}{ref}
\LastPageEnding

\end{document}